\def\marker{\>\hbox{${\vcenter{\vbox{
    \hrule height 0.4pt\hbox{\vrule width 0.4pt height 6pt
    \kern6pt\vrule width 0.4pt}\hrule height 0.4pt}}}$}\>}

\documentclass[12pt]{article}

\usepackage[left=2cm,right=2cm,top=3cm,bottom=3cm]{geometry}

\usepackage{amsmath, amssymb, latexsym, amsthm}
\usepackage{fullpage}

\usepackage{tikz}
\usetikzlibrary{shapes}

\newtheorem{theorem}{Theorem} 
\newtheorem{theorem*}{Theorem} 
\newtheorem{proposition}[theorem]{Proposition} 

\newtheorem{lemma}[theorem]{Lemma}
\newtheorem{conjecture}[theorem]{Conjecture}

\theoremstyle{definition}

\newtheorem{question}{Question}

\newtheorem{construction}{Construction}

\theoremstyle{remark}

{\end{enumerate}}

\newcommand{\FL}[1]{\left\lfloor #1 \right\rfloor}

\newcommand{\set}[2]{\{#1,\ldots,#2\}}

\DeclareMathOperator{\sat}{sat}

\title{Saturation numbers in tripartite graphs}
\author{Eric Sullivan\footnotemark[1] and Paul S.\ Wenger\footnotemark[2]}

\begin{document}

\maketitle

\begin{abstract}
Given graphs $H$ and $F$, a subgraph $G\subseteq H$ is an {\it $F$-saturated subgraph} of $H$ if $F\nsubseteq G$, but $F\subseteq G+e$ for all $e\in E(H)\setminus E(G)$.
The {\it saturation number of $F$ in $H$}, denoted $\sat(H,F)$, is the minimum number of edges in an $F$-saturated subgraph of $H$.
In this paper we study saturation numbers of tripartite graphs in tripartite graphs.
For $\ell\ge 1$ and  $n_1$, $n_2$, and $n_3$ sufficiently large, we determine $\sat(K_{n_1,n_2,n_3},K_{\ell,\ell,\ell})$ and $\sat(K_{n_1,n_2,n_3},K_{\ell,\ell,\ell-1})$ exactly and $\sat(K_{n_1,n_2,n_3},K_{\ell,\ell,\ell-2})$ within an additive constant.
We also include general constructions of $K_{\ell,m,p}$-saturated subgraphs of $K_{n_1,n_2,n_3}$ with few edges for $\ell\ge m\ge p>0$.

\end{abstract}

\renewcommand{\thefootnote}{\fnsymbol{footnote}}
\footnotetext[1]{
Department of Mathematical and Statistical Sciences, University of Colorado Denver, Denver, CO; {\tt eric.2.sullivan@ucdenver.edu}}
\footnotetext[2]{
School of Mathematical Sciences, Rochester Institute of Technology, Rochester, NY;
{\tt pswsma@rit.edu}.}
\renewcommand{\thefootnote}{\arabic{footnote}}

\baselineskip18pt

\section{Introduction}

In this paper, all graphs are simple and we let $V(G)$ and $E(G)$ denote the vertex set and edge set of the graph $G$, respectively.
Let $\overline{G}$ denote the complement of $G$.
For a set of vertices $S\subseteq V(G)$, we let $G[S]$ denote the induced subgraph of $G$ on $S$.

Given a graph $F$, a graph $G$ is {\it $F$-saturated} if $F$ is not a subgraph of $G$ but $F$ is a subgraph of $G+e$ for any edge $e\in E(\overline{G})$.
The {\it saturation number} of $F$ is the minimum size of an $n$-vertex $F$-saturated graph, and is denoted $\sat(n,F)$.
Saturation numbers were first studied by Erd\H os, Hajnal, and Moon~\cite{EHM}, who proved that $\sat(n,K_k)=(k-2)n-\binom{k-1}{2}$ and characterized the $n$-vertex $K_k$-saturated graphs with this number of edges.
For a thorough account of the results known about saturation numbers, the reader should consult the excellent survey of Faudree, Faudree, and Schmitt~\cite{FFS}.

Because saturation numbers consider the addition of any edge from $\overline G$ to $G$, it is natural in this setting to think of $G$ as a subgraph of the complete graph $K_n$.
In this paper we consider saturation numbers when $G$ is treated as a subgraph of a complete tripartite graph.

Let $F$ and $H$ be graphs be fixed graphs; we call $H$ the {\it host graph}.
A subgraph $G$ of $H$ is an {\it $F$-saturated subgraph} of $H$ if $F$ is not a subgraph of $G$, but $F$ is a subgraph of $G+e$ for all $e\in E(H)\setminus E(G)$.
The {\it saturation number of $F$ in $H$} is the minimum number of edges in an $F$-saturated subgraph of of $H$, and is denoted $\sat(H,F)$.
With this notation, $\sat(n,F)=\sat(K_n,F)$.

The first result on saturation numbers in host graphs that are not complete is from a related problem in bipartite graphs.
Let $\sat(K_{(n_1,n_2)},K_{(\ell,m)})$ denote the minimum number of edges in a bipartite $G$ graph on the vertex set $V_1\cup V_2$ where $|V_i|=n_i$ such that: 1) $G$ does not contain $K_{\ell,m}$ with $\ell$ vertices in $V_1$ and $m$ vertices in $V_2$, and 2) the addition of any edge joining $V_1$ and $V_2$ yields a copy of $K_{\ell,m}$ with $\ell$ vertices in $V_1$ and $m$ vertices in $V_2$.
This parameter is the minimization analogue of the Zarankiewicz number.
Bollob\' as and Wessel~\cite{Bol1,Bol2,Wes1,Wes2} independently proved that  $\sat(K_{(n_1,n_2)},K_{(\ell,m)})=(m-1)n_1+(\ell-1)n_2-(m-1)(\ell-1)$ for $2\le \ell\le n_1$ and $2\le m\le n_2$, confirming a conjecture of Erd\H os, Hajnal, and Moon from~\cite{EHM}.

In~\cite{MS}, Moshkovitz and Shapira studied saturation numbers in $d$-uniform $d$-partite hypergraphs.
When $d=2$, this reduces to saturation numbers of bipartite graphs in bipartite graphs.
They provided a construction showing that $\sat(K_{n,n},K_{\ell,m})\le(\ell+m-2)n-\FL{\left(\frac{(\ell+m-2)}{2}\right)^2}$ and conjectured that the bound is sharp for $n$ sufficiently large.
This upper bound shows that for $n$ sufficiently large, $\sat(K_{n,n},K_{\ell,m})< \sat(K_{(n,n)},K_{(\ell,m)})$.
Recently, Gan, Kor\'andi and Sudakov \cite{GKS} showed that $\sat(K_{n,n},K_{\ell,m})\ge (\ell+m-2)n-(\ell+m-2)^2$ and proved that the Moshkovitz-Shapira bound is sharp for $K_{2,3}$, the first nontrivial case.

Let $K_{k}^{n}$ denote the complete $k$-partite graph in which each partite set has order $n$.
In~\cite{FJPW}, Ferrara, Jacobson, Pfender, and the second author studied the saturation number of $K_3$ in balanced multipartite graphs.
They proved that if $k\ge 3$ and $n\ge 100$, then
\[\sat(K_k^n,K_3)=\min\{2kn+n^2-4k-1, 3kn-3n-6\}.\]
Furthermore, they characterized the $K_3$-saturated subgraphs of $K_k^n$ of minimum size.

The focus of this paper is the saturation numbers in complete tripartite graphs.
In Section~\ref{sec:con}, we provide constructions of $K_{\ell,m,p}$-saturated subgraphs of $K_{n_1,n_2,n_3}$ with small size.
In Section~\ref{Sec3}, we determine $\sat(K_{n_1,n_2,n_3},K_{\ell,\ell,\ell})$ and $\sat(K_{n_1,n_2,n_3},K_{\ell,\ell,\ell-1})$ and characterize the $K_{\ell,\ell,\ell}$-saturated subgraphs and $K_{\ell,\ell,\ell-1}$-saturated subgraphs of $K_{n_1,n_2,n_3}$ of minimum size.
In Section~\ref{sec4}, we prove that for $\sat(K_{n,n,n},K_{\ell,\ell,\ell-2})$, the upper bound obtained from the construction in Section~\ref{sec:con} is correct within an additive constant depending on $\ell$.
Finally, Section~\ref{sec5} contains various conjectures and open questions for future work.

Throughout the paper, we will assume that $n_1\ge n_2\ge n_3$, and that the partite sets of $K_{n_1,n_2,n_3}$ are $V_1$, $V_2$, and $V_3$ with $|V_i|=n_i$.
We label the vertices in $V_i$ as $V_i=\{v_i^1,\ldots,v_i^{n_i}\}$.
When $G$ is a tripartite graph on the vertex set $V_1\cup V_2\cup V_3$ we let $\delta_i(G)$ denote the minimum degree of the vertices in $V_i$.
When the graph in question is clear we simply write $\delta_i$.
For a vertex $v\in G$, we let $N_i(v)$ denote the set of neighbors of $v$ in set $V_i$; that is, $N_i(v)=N(v)\cap V_i$.
Similarly, if $S$ is a set of vertices in $G$, then $N_i(S)=\bigcup_{v\in S}N_i(v)$.
Throughout the paper, all arithmetic in subscripts is performed modulo 3.
We also use $[k]$ to denote the set $\{1,\ldots,k\}$.

\section{Constructions of saturated subgraphs of $K_{n_1,n_2,n_3}$}\label{sec:con}

This section contains constructions of $K_{\ell,m,p}$-saturated subgraphs of $K_{n_1,n_2,n_3}$ with few edges.
We begin with two constructions of $K_{\ell,m,p}$-saturated subgraphs of $K_{n_1,n_2,n_3}$ when $m=p$.
The reader is invited to keep in mind the particular case of $K_{\ell,\ell,\ell}$, in which the constructions are greatly simplified and which we prove are best possible in Section~\ref{Sec3}.

\begin{construction}\label{con:1}
Let $\ell$ and $m$ be positive integers such that $\ell\ge m$.
Let $n_1\ge n_2\ge n_3\ge \max\{\ell+2,3\ell-2m-2\}$.
For each $i\in[3]$, let $S_i$ be the $m$-vertex set $\set{v_i^{n_i-m+1}}{v_i^{n_i}}$ and join $S_i$ to $V_{i+1}$, and $V_{i+2}$.
When $\ell>m$, add the following edges, where arithmetic in the superscripts of vertices in $V_i$ is performed modulo $n_i-m$:
\begin{enumerate}
\item for $a\in [n_3-m]$, join $v_3^a$ to $\{v_1^a,\ldots,v_1^{a+\ell-m-1}\}\cup \{v_2^a,\ldots,v_2^{a+\ell-m-1}\}$;
\item for $a\in [n_2-m]$, join $v_2^a$ to $\{v_1^{a+\ell-m},\ldots,v_1^{a+2\ell-2m-1}\}$.
\end{enumerate}
Finally, in all cases, remove the edges $v_1^{n_1}v_2^{n_2}$, $v_1^{n_1}v_3^{n_3}$, and $v_2^{n_2}v_3^{n_3}$ (see Figure~\ref{fig:con1}). 
We call this graph $G_1$.

For a set of integers $S$, let $S\bmod n$ denote the set of residues of the elements of $S$ modulo $n$.
Thus we have 
\begin{align*}
E(G_1)&=\big(\{v_i^rv_j^s : i\in[3],j\in[3],i\neq j,  n_i-m+1\le r\le n_i \text{ or } n_j-m+1\le s\le n_j \}\\
&\qquad\cup\{v_3^{a} v_j^{b}: j\in\{1,2\}, a\in[n_3-m],b\in\set{a}{a+\ell-m-1}\bmod{(n_j-m)}\}\\
&\qquad\cup \{v_2^{a}v_1^{b}: a\in[n_2-m],b\in\set{a+\ell-m}{a+2\ell-2m-1 }\bmod{(n_1-m)}\}\big)\\
&\qquad \setminus  \{v_1^{n_1}v_2^{n_2}, v_1^{n_1}v_3^{n_3},v_2^{n_2}v_3^{n_3} \}.
\end{align*}
\end{construction}

\begin{figure}[h]
\centering
\begin{tikzpicture}[x=.6cm,y=.6cm]
\path
(0,1.75) node [shape=ellipse,minimum width=1.5cm, minimum height=.75cm,draw](V1){$V_1\backslash S_1$}
(0,0) node [shape=circle, minimum width=.75cm , draw](S1){$S_1$}
(0,-8) node (r1){max degree}
(0,-8.7) node (r6){$\ell-m$}
(-6,-1.5) node (r2){max degree}
(-6,-2.2) node (r5){$\ell-m$}
(6,-1.5) node (r3){max degree}
(6,-2.2) node (r4){$\ell-m$}
(-5.5,-5.196) node [shape=ellipse,minimum width=1.5cm, minimum height=.75cm,draw](V3){$V_3\backslash S_3$}
(-3,-5.196) node [shape=circle,minimum width=.75cm , draw](S3){$S_3$}
(3,-5.196) node [shape=circle,minimum width=.75cm , draw](S2){$S_2$}
(5.5,-5.196) node  [shape=ellipse,minimum width=1.5cm, minimum height=.75,draw](V2){$V_2\backslash S_2$}

(0,-1.7) node [draw=none](n1){$v_1^{n_1}$}
(1.4,-4.3) node [draw=none](n2){$v_2^{n_2}$}
(-1.5,-4.3) node [draw=none](n3){$v_3^{n_3}$};
\fill (0,-.75) circle (3pt)
(2.4,-4.75)circle (3pt) 
(-2.4,-4.75)circle (3pt);
\draw (V2) to [bend right] (S1) (S1) to [bend right] (V3)
        (V2)to [bend left] (S3) (S3) to [bend left] (V1) (V1)to [bend left](S2)(S2)to [bend left](V3);
\draw       
 (S3)to [bend left](S1)(S1)to [bend left](S2)(S2)to [bend left](S3);
\draw (V3) to [bend right=24] (r1) (r1) to [bend right=24] (V2);
\draw (V3) to [bend left=15] (r5) (r2) to [bend left=22] (V1);
\draw (V2) to [bend right=15] (r4) (r3) to [bend right=22] (V1);
\draw [dotted, thick] (0,-.75) to (2.4,-4.75) to (-2.4,-4.75) to(0,-.75);
\end{tikzpicture}
\caption[Construction~\ref{con:1}, a $K_{\ell,m,m}$-saturated subgraph of $K_{n_1,n_2,n_3}$]{\label{fig:con1}Construction~\ref{con:1}: A $K_{\ell,m,m}$-saturated subgraph of $K_{n_1,n_2,n_3}$.  Solid lines denote complete joins between sets, and dotted lines denote edges that have been removed.  The lines marked with ``max degree $\ell-m$" represent the edges described in items 1 and 2 of Construction~\ref{con:1}.}
\end{figure}
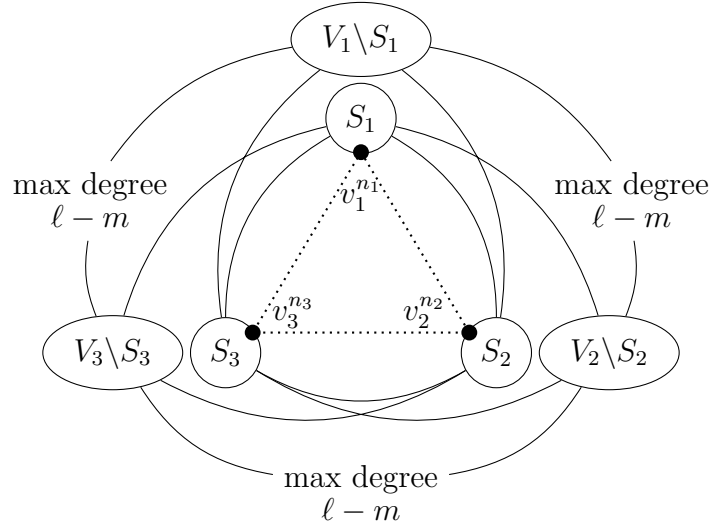

For the particular case of $K_{1,1,1}$, Construction~\ref{con:1} reduces to the obvious extension of the tripartite case of Construction~2 from~\cite{FJPW}.

Our next construction describes a family of three $K_{\ell,m,p}$-saturated subgraphs of $K_{n_1,n_2,n_3}$ for the case when $m=p$.
It is a very slight modification of Construction~\ref{con:1}.

\begin{construction} \label{con:2}  
For $i\in [3]$, let $G_2^i$ be the graph obtained from the graph from Construction~\ref{con:1} by removing the set $\{v_i^{n_i}v_{i+1}^{n_{i+1}}, v_i^{n_i-1}v_{i+2}^{n_{i+2}}, v_{i+1}^{n_{i+1}}v_{i+2}^{n_{i+2}}\}$
instead of $\{v_1^{n_1}v_2^{n_2},v_1^{n_1}v_3^{n_3},v_2^{n_2}v_3^{n_3}\}$ (see Figure~\ref{fig:con2}). 
\end{construction}
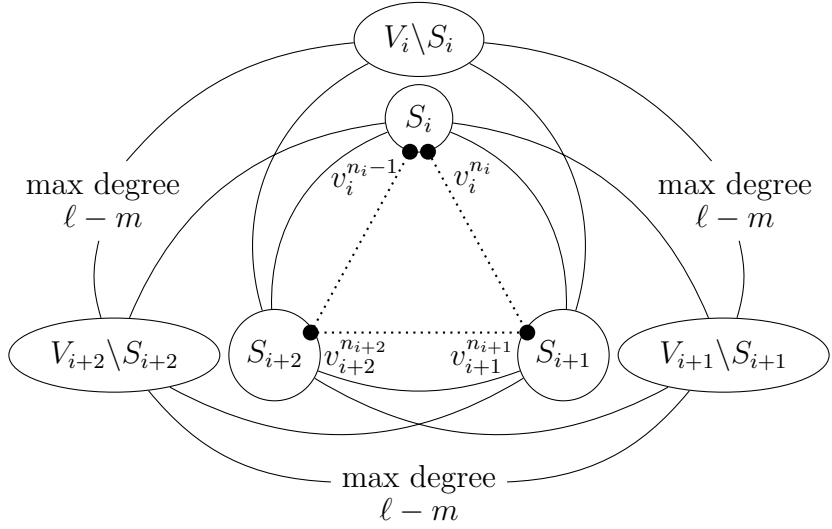
\begin{figure}[h]
\centering
\begin{tikzpicture}[x=.6cm,y=.6cm]
\path
(0,1.75) node [shape=ellipse,minimum width=1.5cm, minimum height=.75cm,draw](V1){$V_i\backslash S_i$}
(0,0) node [shape=circle, minimum width=.75 cm, draw](S1){$S_i$}
(0,-8) node (r1){max degree}
(0,-8.7) node (r6){$\ell-m$}
(-7,-1.5) node (r2){max degree}
(-7,-2.2) node (r5){$\ell-m$}
(7,-1.5) node (r3){max degree}
(7,-2.2) node (r4){$\ell-m$}
(-6.75,-5.25) node [shape=ellipse,minimum width=1.5cm, minimum height=.75cm,draw](V3){$V_{i+2}\backslash S_{i+2}$}
(-3.2,-5.25) node [shape=circle,minimum width=.75 cm, draw](S3){$S_{i+2}$}
(3.2,-5.25) node [shape=circle,minimum width=.75 cm, draw](S2){$S_{i+1}$}
(6.75,-5.25) node  [shape=ellipse,minimum width=1.5cm, minimum height=.75cm,draw](V2){$V_{i+1}\backslash S_{i+1}$}

(-1.2,-1.3) node[draw=none, ](n1){$v_i^{n_i-1}$}
(1.2,-1.3) node [draw=none,](n1){$v_i^{n_i}$}
(1.4,-5.3) node [draw=none](n2){$v_{i+1}^{n_{i+1}}$}
(-1.4,-5.3) node [draw=none](n3){$v_{i+2}^{n_{i+2}}$};
\fill (.2,-.74)circle (3pt)
(-.2,-.74) circle (3pt)
(2.4,-4.75)circle (3pt) 
(-2.4,-4.75)circle (3pt);
\draw (V2) to [bend right] (S1) (S1) to [bend right] (V3)
        (V2)to [bend left] (S3) (S3) to [bend left=40] (V1) (V1)to [bend left=40](S2)(S2)to [bend left](V3);
\draw       
 (S3)to [bend left=35](S1)(S1)to [bend left=35](S2)(S2)to [bend left=20](S3);
\draw [dotted, thick] (.2,-.74) to (2.4,-4.75) to (-2.4,-4.75) to(-.2,-.74);
\draw (V3) to [bend right=24] (r1) (r1) to [bend right=24] (V2);
\draw (V3) to [bend left=15] (r5) (r2) to [bend left=22] (V1);
\draw (V2) to [bend right=15] (r4) (r3) to [bend right=22] (V1);
\end{tikzpicture}
\caption[Construction~\ref{con:2}, A $K_{\ell,m,m}$-saturated subgraph of $K_{n_1,n_2,n_3}$]{\label{fig:con2}Construction~\ref{con:2}: A $K_{\ell,m,m}$-saturated subgraph of $K_{n_1,n_2,n_3}$.  Solid lines denote complete joins between sets, and dotted lines denote edges that have been removed.
The lines marked with ``max degree $\ell-m$" represent the edges described in items 1 and 2 of Construction~\ref{con:1}.}
\end{figure}

\begin{theorem}\label{thm:lmmupperbound}
Let $\ell$ and $m$ be positive integers such that $\ell\ge m$.
For $n_1\ge n_2\ge n_3\ge \max\{\ell+2,3\ell-2m-1\}$, the graphs from Construction~\ref{con:1} and Construction~\ref{con:2} are $K_{\ell,m,m}$-saturated subgraphs of $K_{n_1,n_2,n_3}$.
Thus, \[ \sat(K_{n_1,n_2,n_3},K_{\ell,m,m})\le 2m(n_1+n_2+n_3)+(\ell-m)(n_2+2n_3)-3\ell m-3.\]
\end{theorem}

\begin{proof}
Let $G$ be a graph from Construction~\ref{con:1} or~\ref{con:2}.
By construction, $G-(S_1\cup S_2\cup S_3)$ is triangle-free.
Therefore, if $v\in V_i\setminus S_i$, then $G[N(v)]$ does not contain $K_{\ell,m}$ as a subgraph.
Since $G[S_i\cup S_{i+1}]$ is not a complete bipartite graph, it then follows that $G$ is $K_{\ell,m,m}$-free.

Let $e=uv$ be a nonedge in $G$.
We show that $G+e$ contains $K_{\ell,m,m}$; there are two cases to consider.

{\bf Case 1:} {\it $e$ joins two vertices in $S_1\cup S_2\cup S_3$.}
If $e$ joins $S_i$ and $S_{i+1}$, then $G+e$ contains $K_{\ell,m,m}$ on the vertices $\{v_{i+2}^1,\ldots,v_{i+2}^{\ell}\}\cup S_i\cup S_{i+1}$.

{\bf Case 2:} {\it $e$ joins two vertices in $V(G)\setminus(S_1\cup S_2\cup S_3)$.}
Let $i,j\in[3]$ such that $i<j$, and assume that $e = v_j^av_i^b$ where $a\in [n_j-m]$ and $b\in [n_i-m]$.
Let $k$ be the third value in $[3]$.
Let $x_i\in S_i$ and $x_j\in S_j$ be the vertices that have a nonneighbor in $S_k$.
By construction, $S_i-x_i$ is completely joined to $S_j-x_j$.
In this case, $G+e$ contains $K_{\ell,m,m}$ on the vertex set $(N_i(v_j^a)+v_i^b-x_i)\cup (S_j+v_j^a-x_j)\cup S_k$.
\end{proof}

We now construct $K_{\ell,m,p}$-saturated subgraphs of $K_{n_1,n_2,n_3}$ when $m>p$.
Like Constructions~\ref{con:1} and~\ref{con:2}, the subgraph of this construction induced by $(V_1\setminus S_1)\cup (V_2\setminus S_2)\cup (V_3\setminus S_3)$ consists of bipartite graphs with maximum degree $\ell-m$.
Unlike Constructions~\ref{con:1} and~\ref{con:2}, the vertices in this set have fewer than $\ell$ neighbors in the other partite sets.
Therefore it is not necessary to specify completely the neighborhoods of these vertices.

\begin{construction}
\label{con:3}
Let $\ell$, $m$, and $p$ be positive integers such that $\ell\ge m>p$.
Let $n_1\ge n_2\ge n_3\ge \ell$.
For each $i\in [3]$ let $S_i$ be an $(m-1)$-vertex subset of $V_i$ and join $S_i$ to $V_{i+1}$ and $V_{i+2}$.
For $i<j$, join $V_i\setminus S_i$ to $V_j\setminus S_j$ with an $(\ell-m)(n_j-m+1)$-edge graph with maximum degree $\ell-m$.
Thus each vertex in $V_j\setminus S_j$ has exactly $\ell-m$ neighbors in $V_i\setminus S_i$, and each vertex in $V_i\setminus S_i$ has at most $\ell-m$ neighbors in $V_j\setminus S_j$.
\end{construction}

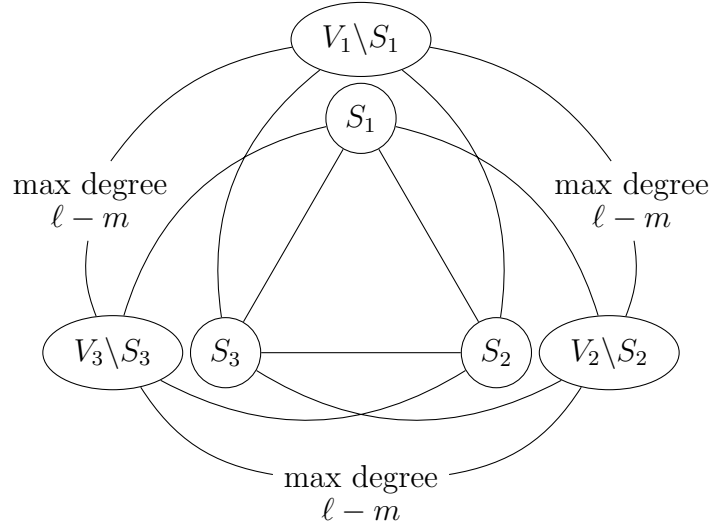
\begin{figure}[h]
\centering
\begin{tikzpicture}[x=.6cm,y=.6cm]
\path
(0,1.75) node [shape=ellipse,minimum width=1.5cm, minimum height=.75cm,draw](V1){$V_1\backslash S_1$}
(0,0) node [shape=circle, minimum width=.75cm , draw](S1){$S_1$}
(0,-8) node (r1){max degree}
(0,-8.7) node (r6){$\ell-m$}
(-6,-1.5) node (r2){max degree}
(-6,-2.2) node (r5){$\ell-m$}
(6,-1.5) node (r3){max degree}
(6,-2.2) node (r4){$\ell-m$}
(-5.5,-5.196) node [shape=ellipse,minimum width=1.5cm, minimum height=.75cm,draw](V3){$V_3\backslash S_3$}
(-3,-5.196) node [shape=circle,minimum width=.75cm , draw](S3){$S_3$}
(3,-5.196) node [shape=circle,minimum width=.75cm , draw](S2){$S_2$}
(5.5,-5.196) node  [shape=ellipse,minimum width=1.5cm, minimum height=.75,draw](V2){$V_2\backslash S_2$}

;
\draw (V2) to [bend right] (S1) (S1) to [bend right] (V3)
        (V2)to [bend left] (S3) (S3) to [bend left] (V1) (V1)to [bend left](S2)(S2)to [bend left](V3);
\draw       
 (S3)to (S1)(S1)to  (S2)(S2)to  (S3);
\draw (V3) to [bend right=24] (r1) (r1) to [bend right=24] (V2);
\draw (V3) to [bend left=15] (r5) (r2) to [bend left=22] (V1);
\draw (V2) to [bend right=15] (r4) (r3) to [bend right=22] (V1);
;
\end{tikzpicture}
\caption[Construction~\ref{con:3}, a $K_{\ell,m,p}$-saturated subgraph of $K_{n_1,n_2,n_3}$]{\label{fig:con3}Construction~\ref{con:3}: A $K_{\ell,m,p}$-saturated subgraph of $K_{n_1,n_2,n_3}$ for $m>p$.  Solid lines denote complete joins between sets.
The lines marked with ``max degree $\ell-m$" represent the $(\ell-m)(n_j-m+1)$-edge graphs with maximum degree $\ell-m$ used in Construction~\ref{con:3}.}
\end{figure}

\begin{theorem}\label{thm:lmpupperbound}
Let $\ell$, $m$, and $p$ be positive integers such that $\ell\ge m>p$.
For $n_1\ge n_2\ge n_3\ge \ell$, the graph from Construction~\ref{con:3} is a $K_{\ell,m,p}$-saturated subgraph of $K_{n_1,n_2,n_3}$.
Thus, \[ \sat(K_{n_1,n_2,n_3},K_{\ell,m,p})\le 2(m-1)(n_1+n_2+n_3)+(\ell-m)(n_2+2n_3)-3\ell(m-1)+3m-3.\]
\end{theorem}

\begin{proof}
Let $G$ be the graph described in Construction~\ref{con:3}.
Let $i\in [3]$.
If $v\in V_i\setminus S_i$, then $v$ has at most $\ell-1$ neighbors in $V_{i+1}$ and at most $\ell-1$ neighbors in $V_{i+2}$.
Since there are only $m-1$ vertices in $S_i$, it follows that $G$ does not contain $K_{\ell,m}$, and therefore $G$ is $K_{\ell,m,p}$-free.

Let $i,j\in[3]$ such that $i<j$, and let $k$ be the third value in $[3]$.
Let $e$ be a nonedge in $G$ joining $v_i\in V_i$ and $v_j\in V_j$.
Thus $G+e$ contains $K_{\ell,m,m-1}$ on the vertex set $(N_i(v_j)+v_i)\cup (S_j+v_j)\cup S_k$.
Since $m>p$, it follows that $G+e$ contains $K_{\ell,m,p}$.
\end{proof}

We include two final constructions in the special case of $K_{\ell,m,p}$-saturated subgraphs of $K_{n,n,n}$.
These constructions are inspired by the $K_{\ell,m}$-saturated subgraphs of $K_{n,n}$ used in~\cite{MS} and~\cite{GKS}.
When the host graph is balanced, Constructions~\ref{con:1},~\ref{con:2}, and~\ref{con:3} contain large $(\ell-m)$-regular graphs; we will replace those graphs with graphs with slightly fewer edges.

\begin{construction}
\label{con:4}
Let $\ell$ and $m$ be positive integers such that $\ell\ge m$ and let $$n\ge \max\left\{\ell+2,3\ell+\FL{\frac{\ell-m}{2}}-2m-2\right\}.$$
For each $i\in [3]$, let $S_i=\{v_i^1,\ldots,v_i^m\}$ and join $S_i$ to $V_{i+1}$ and $V_{i+2}$.
Let $t=\FL{\frac{\ell-m}{2}}$, and for each $i\in [3]$ let $T_i=\{v_i^{m+1},\ldots,v_i^{m+t}\}$.
For all $i\in[3]$, completely join $T_i$ to $T_{i+1}$.
Let $\bigcup_{i\in[3]}(V_i\setminus(S_i\cup T_i))$ span a triangle-free tripartite graph so that for all $i\in[3]$, each vertex in $V_i\setminus(S_i\cup T_i)$ has exactly $\ell-m$ neighbors in both $V_{i+1}\setminus(S_{i+1}\cup T_{i+1})$ and $V_{i+2}\setminus(S_{i+2}\cup T_{i+2})$ (such a graph is easily obtained using items 1 and 2 from Construction~\ref{con:1}).
Finally, remove the edges $\{v_1^1v_2^1,v_1^1v_3^1,v_2^1v_3^1\}$ (see Figure~\ref{figure:con4}).
\end{construction}

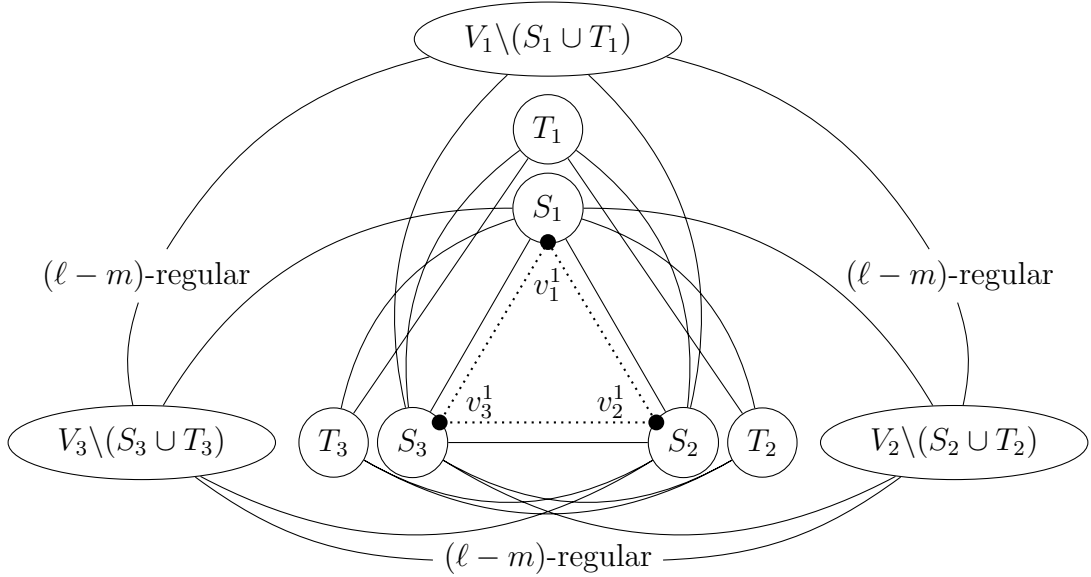
\begin{figure}[h]
\centering
\begin{tikzpicture}[x=.6cm,y=.6cm]
\path
(0,3.75) node [shape=ellipse,minimum width=1.5cm,minimum height=.75cm,draw](V1){$V_1\backslash (S_1\cup T_1)$}
(0,0) node [shape=circle,minimum width=.75 cm,draw](v1){$S_1$}
(-9,-5.196) node [shape=ellipse,minimum width=1.5cm,minimum height=.75cm,draw](V3){$V_3\backslash (S_3\cup T_3)$}
(-3,-5.196) node [shape=circle,minimum width=.75 cm,draw](v3){$S_3$}
(9,-5.196) node [shape=ellipse,minimum width=1.5cm,minimum height=.75cm,draw](V2){$V_2\backslash (S_2\cup T_2)$}
(3,-5.196) node [shape=circle,minimum width=.75 cm,draw](v2){$S_2$}
(-4.75,-5.196) node [shape=circle,minimum width=.75 cm,draw](k3){$T_3$}
(4.75,-5.196) node [shape=circle,minimum width=.75 cm,draw](k2){$T_2$}
(0,1.75) node [shape=circle,minimum width=.75 cm,draw](k1){$T_1$}
(0,-7.75) node (r1){$(\ell-m)$-regular}
(-8.9,-1.5) node (r2){$(\ell-m)$-regular}
(8.9,-1.5) node (r3){$(\ell-m)$-regular}
(0,-1.7) node [draw=none](n1){$v_1^{1}$}
(1.4,-4.3) node [draw=none](n2){$v_2^{1}$}
(-1.5,-4.3) node [draw=none](n3){$v_3^{1}$};
\fill (0,-.75) circle (3pt)
(2.4,-4.75)circle (3pt) 
(-2.4,-4.75)circle (3pt);
\draw(V1) to [bend left] (v2)(v2) to [bend left] (V3) (V3) to [bend left] (v1) (v1) to [bend left] (V2) (V2) to [bend left] (v3)(v3) to [bend left] (V1) (v1)--(v2)(v2)--(v3)(v3)--(v1)
(k3)to [bend right] (k2)(k2)--(k1)--(k3)
(v1) to [bend left](k2)(k2)to [bend left](v3)(v3)to [bend left](k1)(k1)to [bend left](v2)(v2) to [bend left](k3)(k3)to [bend left](v1)
;
\draw (V3) to [bend right=17] (r1) (r1) to [bend right=17] (V2);
\draw (V3) to [bend left=15] (r2) (r2) to [bend left=22] (V1);
\draw (V2) to [bend right=15] (r3) (r3) to [bend right=22] (V1);
\draw [dotted, thick] (0,-.75) to (2.4,-4.75) to (-2.4,-4.75) to(0,-.75);
\end{tikzpicture}
\caption[Construction~\ref{con:4}, a $K_{\ell,m,m}$-saturated subgraph of $K_{n,n,n}$]{\label{figure:con4}Construction~\ref{con:4}: A $K_{\ell,m,m}$-saturated subgraph of $K_{n,n,n}$. 
Solid lines denote complete joins between sets, and dotted lines denote edges that have been removed.  The lines marked with ``$(\ell-m)$-regular" represent the triangle-free tripartite graph used in Construction~\ref{con:4}.}
\end{figure}

It is possible to modify Construction~\ref{con:4} so that the edges removed induce $P_4$ rather than $K_3$ as in Construction~\ref{con:2} (for instance, remove $\{v_i^1v_{i+1}^1,v_i^2v_{i+2}^1,v_{i+1}^1v_{i+2}^1\}$).
Since we do not prove that these constructions are best possible nor that they characterize the $K_{\ell,m,m}$-saturated subgraphs of $K_{n,n,n}$ of minimum size, we do not include this variant as a separate construction.

We now present a $K_{\ell,m,p}$-saturated subgraph of $K_{n,n,n}$ for $m>p$.

\begin{construction}
\label{con:5}
Let $\ell$, $m$, and $p$ be positive integers such that $\ell\ge m>p$ and let $n\ge \ell+\FL{\frac{\ell-m}{2}}-1$.
For each $j\in [3]$, let $S_j$ be an $(m-1)$-vertex subset of $V_j$ and join $S_i$ to $V_{i+1}$ and $V_{i+2}$.
Let $t=\FL{\frac{\ell-m}{2}}$, and for each $j\in [3]$ let $T_i$ be a $t$-vertex subset of $V_j\setminus  S_j$.
For all $i\in[3]$, completely join $T_i$ to $T_{i+1}$.
For each $i\in [3]$, let $(V_i\cup V_{i+1})\setminus (S_i\cup S_{i+1}\cup T_i\cup T_{i+1})$ induce an $(\ell-m)$-regular bipartite graph.  
\end{construction}

Constructions~\ref{con:4} and~\ref{con:5} yield the following two theorems.
The proofs of these theorems follow almost immediately from the proofs of Theorems~\ref{thm:lmmupperbound} and~\ref{thm:lmpupperbound}, respectively, and therefore we omit them.

\begin{theorem}
Let $\ell$ and $m$ be positive integers such that $\ell\ge m$ and let $$n\ge \max\left\{\ell+2,3\ell+\FL{\frac{\ell-m}{2}}-2m-2\right\}.$$
The graph from Construction~\ref{con:4} is a $K_{\ell,m,m}$-saturated subgraph of $K_{n,n,n}$, and thus 
\[\sat(K_{n,n,n},K_{\ell,m,p})\le 3(\ell+m)n-3\left(\ell-m-\FL{\frac{\ell-m}{2}}\right)\FL{\frac{\ell-m}{2}}-3\ell m-3.\]
\end{theorem}

\begin{theorem}\label{thrm:klm}
Let $\ell$, $m$, and $p$ be positive integers such that $\ell\ge m>p$ and let $n\ge \ell+\FL{\frac{\ell-m}{2}}-1$.
The graph from Construction~\ref{con:5} is a $K_{\ell,m,p}$-saturated subgraph of $K_{n,n,n}$, and thus 
\[\sat(K_{n,n,n},K_{\ell,m,p})\le 3(\ell+m-2)n-3(m-1)(\ell-1)+3\FL{\frac{\ell-m}{2}}^2-3(\ell-m)\FL{\frac{\ell-m}{2}}.\]
\end{theorem}

\begin{figure}[h]
\centering
\begin{tikzpicture}[x=.6cm,y=.6cm]
\path
(0,3.75) node [shape=ellipse,minimum width=1.5cm,minimum height=.75cm,draw](V1){$V_1\backslash (S_1\cup K_1)$}
(0,0) node [shape=circle,minimum width=.75 cm,draw](v1){$S_1$}
(-9,-5.196) node [shape=ellipse,minimum width=1.5cm,minimum height=.75cm,draw](V3){$V_3\backslash (S_3\cup K_2)$}
(-3,-5.196) node [shape=circle,minimum width=.75 cm,draw](v3){$S_3$}
(9,-5.196) node [shape=ellipse,minimum width=1.5cm,minimum height=.75cm,draw](V2){$V_2\backslash (S_2\cup K_3)$}
(3,-5.196) node [shape=circle,minimum width=.75 cm,draw](v2){$S_2$}
(-4.75,-5.196) node [shape=circle,minimum width=.75 cm,draw](k3){$T_3$}
(4.75,-5.196) node [shape=circle,minimum width=.75 cm,draw](k2){$T_2$}
(0,1.75) node [shape=circle,minimum width=.75 cm,draw](k1){$T_1$}
(0,-7.75) node (r1){$(\ell-m)$-regular}
(-8.9,-1.5) node (r2){$(\ell-m)$-regular}
(8.9,-1.5) node (r3){$(\ell-m)$-regular};
\draw(V1) to [bend left] (v2)(v2) to [bend left] (V3) (V3) to [bend left] (v1) (v1) to [bend left] (V2) (V2) to [bend left] (v3)(v3) to [bend left] (V1) (v1)--(v2)(v2)--(v3)(v3)--(v1)
(k3)to [bend right] (k2)(k2)--(k1)--(k3)
(v1) to [bend left](k2)(k2)to [bend left](v3)(v3)to [bend left](k1)(k1)to [bend left](v2)(v2) to [bend left](k3)(k3)to [bend left](v1)
;
\draw (V3) to [bend right=17] (r1) (r1) to [bend right=17] (V2);
\draw (V3) to [bend left=15] (r2) (r2) to [bend left=22] (V1);
\draw (V2) to [bend right=15] (r3) (r3) to [bend right=22] (V1);
\end{tikzpicture}
\caption[Construction~\ref{con:5}, a $K_{\ell,m,p}$-saturated subgraph of $K_{n,n,n}$]{Construction~\ref{con:5}: A $K_{\ell,m,p}$-saturated subgraph of $K_{n,n,n}$.  Solid lines denote complete joins between sets.
The lines marked with ``$(\ell-m)$-regular" represent the $(\ell-m)$-regular bipartite graphs used in Construction~\ref{con:5}.}
\end{figure}
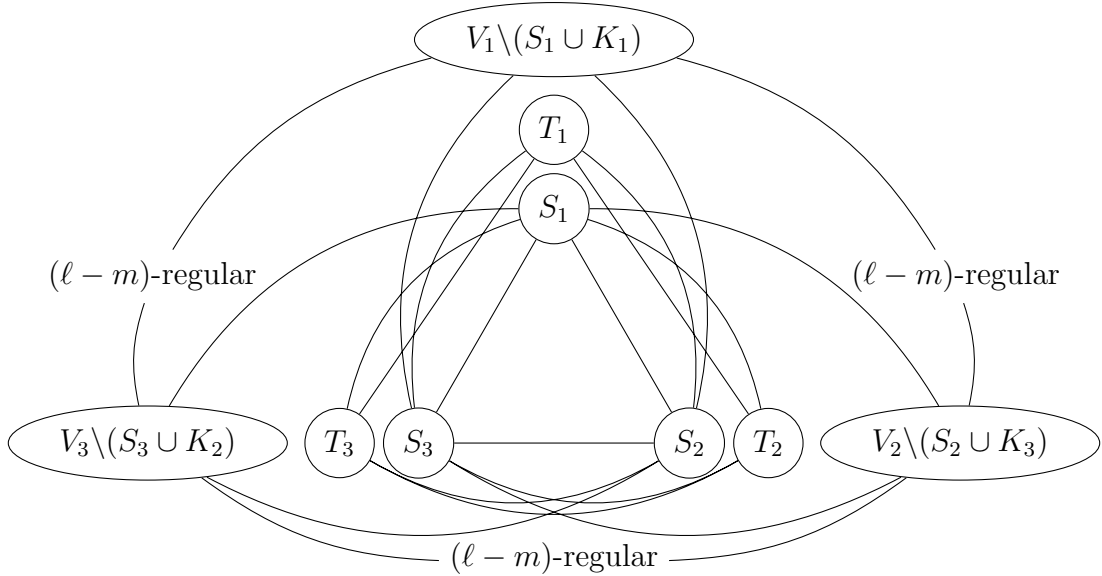

\section{The saturation numbers of $K_{\ell,\ell,\ell}$ and $K_{\ell,\ell,\ell-1}$}\label{Sec3}

In this section we prove the following two theorems on saturation numbers in tripartite graphs.

\begin{theorem}\label{thm:lll}
Let $\ell$ be a positive integer.
If $n_1$, $n_2$, and $n_3$ are positive integers  such that $n_1\ge n_2\ge n_3\ge 32\ell^3+40\ell^2+11\ell$, then
$$sat(K_{n_1,n_2,n_3},K_{\ell,\ell,\ell})=2\ell (n_1+n_2+n_3)-3\ell^2-3.$$
Furthermore, the graphs from Constructions~\ref{con:1} and~\ref{con:2} are the only $K_{\ell,\ell,\ell}$-saturated subgraphs of $K_{n_1,n_2,n_3}$ with this number of edges.
\end{theorem}

\begin{theorem}\label{thm:lll-1}
Let $\ell$ be a positive integer.
If $n_1$, $n_2$, and $n_3$ are positive integers  such that $n_1\ge n_2\ge n_3\ge 32(\ell-1)^3+40(\ell-1)^2+11(\ell-1)$, then
$$sat(K_{n_1,n_2,n_3},K_{\ell,\ell,\ell-1})=2(\ell-1)(n_1+n_2+n_3)-3(\ell-1)^2.$$
Furthermore, the graph from Construction~\ref{con:3} is the unique $K_{\ell,\ell,\ell-1}$-saturated subgraph of $K_{n_1,n_2,n_3}$ with this number of edges.
\end{theorem}

Though $K_{\ell,\ell,\ell}$ and $K_{\ell,\ell,\ell-1}$ correspond to different constructions from Section~\ref{sec:con}, they are both of the form $K_{\ell,\ell,m}$ for $\ell\ge m$.
Thus we begin by establishing some common lemmas on the number of edges in $K_{\ell,\ell,m}$-saturated subgraphs of $K_{n_1,n_2,n_3}$ when $m\ge1$.

\begin{lemma}\label{lem:deltai}
Let $i\in[3]$ and assume that $n_i\ge (3m+1)(\delta_{i+1}+\delta_{i+2})+2m^2+m$. 
If $G$ is a $K_{\ell,\ell,m}$-saturated subgraph of $K_{n_1,n_2,n_3}$ such that $\delta_i>2m$, then $|E(G)|\ge 2m(n_1+n_2+n_3)$.
\end{lemma}

\begin{proof}
For each $j\in[3]$, let $v_j$ be a vertex of degree $\delta_j$ in $V_j$.
Each nonneighbor of $v_i$ in $V_{i+1}\cup V_{i+2}$ must have at least $m$ common neighbors with $v_i$.
Therefore there are at least $m(n_{i+1}+n_{i+2}-\delta_i)$ edges joining $V_{i+1}$ and $V_{i+2}$.
Similarly, there are at least $m(n_{i+1}-\delta_{i+2})$ edges joining $V_{i+1}$ and $N_i(v_{i+2})$ and at least $m(n_{i+2}-\delta_{i+1})$ edges joining $V_{i+2}$ and $N_i(v_{i+1})$.
Finally, there are at least $\delta_i(n_i-\delta_{i+1}-\delta_{i+2})$ edges incident to $V_i\setminus(N_i(v_{i+1})\cup N_i(v_{i+2}))$.
Summing, we have
\begin{align*}
|E(G)|&\ge m(2n_{i+1}+2n_{i+2}-\delta_{i+1}-\delta_{i+2})+\delta_i(n_i-\delta_{i+1}-\delta_{i+2}-m).
\end{align*}
Since $n_i>\delta_{i+1}+\delta_{i+2}+m$, this lower bound is increasing in $\delta_i$.
Therefore, if $\delta_i>2m$, then
\begin{align*}
|E(G)|&\ge m(2n_{i+1}+2n_{i+2}-\delta_{i+1}-\delta_{i+2})+(2m+1)(n_i-\delta_{i+1}-\delta_{i+2}-m)\\
&\ge 2m(n_1+n_2+n_3)+n_i-\left[(3m+1)(\delta_{i+1}+\delta_{i+2})+2m^2+m\right]\\
&\ge 2m(n_1+n_2+n_3).\qedhere
\end{align*}
\end{proof}

\begin{lemma}\label{lem:mindeg}
Let $n_1\ge n_2\ge n_3\ge 32m^3+40m^2+11m$.
If $G$ is a $K_{\ell,\ell,m}$-saturated subgraph of $K_{n_1,n_2,n_3}$ such that $\delta_i>2m$ for some $i\in[3]$, then $|E(G)|\ge 2m(n_1+n_2+n_3)$.
\end{lemma}

\begin{proof}
First observe that each vertex in $V_i$ has at least $m$ neighbors in both $V_{i+1}$ and $V_{i+2}$ or is completely joined to $V_{i+1}$ or $V_{i+2}$.
Thus $\delta(G)\ge 2m$.
There are two cases to consider depending on the order of $n_1$.

{\bf Case 1:} $n_1<4mn_2$.
If $\delta_1\ge 6m$, then $|E(G)|\ge 6mn_1\ge 2m(n_1+n_2+n_3)$.
Therefore we may assume that $\delta_1< 6m$.
If $\delta_2\ge 8m^2+4m$, then 
$|E(G)|\ge (8m^4+4m)n_2 \ge 2m(n_1+n_2+n_3)$.
Therefore we may assume that $\delta_2< 8m^2+4m$.
Since $n_3\ge (3m+1)(8m^2+10m)+2m^2+m$, Lemma~\ref{lem:deltai} implies that if $\delta_3>2m$, then $|E(G)|\ge 2m(n_1+n_2+n_3)$.
Therefore we may assume that $\delta_3=2m$.
Lemma~\ref{lem:deltai} now implies that if $\delta_1>2m$ or $\delta_2>2m$, then $|E(G)|\ge 2m(n_1+n_2+n_3)$.

{\bf Case 2:} $n_1>4mn_2$.
If $\delta_1>2m$, then $|E(G)|\ge (2m+1)n_1\ge 2m(n_1+n_2+n_3)$.
Therefore we may assume that $\delta_1=2m$.
Let $R$ be the set of vertices in $V_1$ with degree $2m$.
If $|V_1\setminus R|\ge 2m(n_2+n_3)$, then $|E(G)|\ge 2m(n_1+n_2+n_3)$.
Therefore we assume that $|V_1\setminus R|< 2m(n_2+n_3)$.

If $v\in R$, then each vertex in $N_2(v)$ is adjacent to every vertex in $V_3\setminus N_3(v)$.
Thus each vertex in $N_2(R)$ has at least $n_3-m$ neighbors in $V_3$.
If $|N_2(R)|\ge 4mn_2/(n_3-m)$, then there are at least $4mn_2$ edges joining $V_2$ and $V_3$, and consequently $|E(G)|\ge 2m(n_1+n_2+n_3)$.
Therefore we may assume that $|N_2(R)|< 4mn_2/(n_3-m)$.

There are at least $\delta_2(n_2-4mn_2/(n_3-m))$ edges incident to $V_2\setminus N_2(R)$.
There are at least $2m(n_1-2m(n_2+n_3))$ edges incident to $R$.
Therefore, if $\delta_2\ge 8m^2+4m+1$, then
\begin{align*}
|E(G)|&\ge 2m(n_1-2m(n_2+n_3))+(8m^2+4m+1)\left(n_2-\frac{4mn_2}{n_3-m}\right)\\
&\ge 2mn_1+4mn_2+n_2-n_2\left(\frac{(8m^2+4m+1)(4m)}{n_3-m}\right)\\
&\ge 2m(n_1+n_2+n_3).
\end{align*}
Therefore we may assume that $\delta_2 \le 8m^2+4m$.

Since $\delta_1=2m$, $\delta_2\le 8m^2+4m$, and $n_3\ge (3m+1)(8m^2+6m)+2m^2+m$, Lemma~\ref{lem:deltai} implies that if $\delta_3>2m$, then $|E(G)|\ge 2m(n_1+n_2+n_3)$.
Therefore we may assume that $\delta_3=2m$.
It now follows from Lemma~\ref{lem:deltai} that if $\delta_2>2m$, then $|E(G)|\ge 2m(n_1+n_2+n_3)$.
\end{proof}

We now prove Theorems~\ref{thm:lll} and~\ref{thm:lll-1}.

\begin{proof}[Proof of Theorem~\ref{thm:lll}]
Let $G$ be a $K_{\ell,\ell,\ell}$-saturated subgraph of $K_{n_1,n_2,n_3}$ of minimum size.
It follows from Lemma~\ref{lem:mindeg} that if $\delta_i>2\ell$ for any $i\in[3]$, then $|E(G)|\ge 2\ell(n_1+n_2+n_3)$.
Since it is clear that $\delta(G)\ge 2\ell$, we assume that $\delta_1=\delta_2=\delta_3=2\ell$.

For $i\in [3]$, let $v_i\in V_i$ be a vertex of degree $2\ell$. 
Thus $v_i$ has $\ell$ neighbors in $V_{i+1}$ and $\ell$ neighbors in $V_{i+2}$, and $G$ contains all edges joining $N_{i+1}(v_i)$ to $V_{i+2}\setminus N_{i+2}(v_i)$ and all edges joining $N_{i+2}(v_i)$ to $V_{i+1}\setminus N_{i+1}(v_i)$.
Therefore, the vertices of degree $2\ell$ in $G$ form an independent set.
Let $S=N(v_1)\cup N(v_2)\cup N(v_3)$ and let $S_i=S\cap V_i$.  
Since $v_{i+1}$ and $v_{i+2}$ have $\ell$ common neighbors, we conclude that $N_i(v_{i+1})=N_i(v_{i+2})$ and therefore $|S_i|=\ell.$
Since the addition of an edge joining $v_i$ and any vertex in $(V_{i+1}\cup V_{i+2})\setminus N(v_i)$ completes a copy of $K_{\ell,\ell,\ell}$, there are at least $\ell^2-1$ edges joining $S_{i+1}$ and $S_{i+2}$.
Therefore there are at least $\ell(n_{i+1}+n_{i+2})-\ell^2-1$ edges joining $V_{i+1}$ and $V_{i+2}$.
Thus $|E(G)|\ge 2\ell(n_1+n_2+n_3)-3\ell^2-3$, and in conjunction with Theorem~\ref{thm:lmmupperbound} we conclude that $\sat(K_{n_1,n_2,n_3},K_{\ell,\ell,\ell})=2\ell(n_1+n_2+n_3)-3\ell^2-3$.

Since $|E(G)|=2\ell(n_1+n_2+n_3)-3\ell^2-3$, it follows that there are exactly $\ell^2-1$ edges joining $S_i$ and $S_{i+1}$ for all $i\in[3]$.
Suppose that $G$ is not isomorphic to a graph from Construction~\ref{con:1} or ~\ref{con:2}.
Thus the three nonedges in $G[S]$ do not induce $K_3$ or $P_4$.
Without loss of generality, assume that $u_i^1u_{i+1}^1$ is a nonedge in $G[S]$ and the other two nonedges in $G[S]$ are incident to $u_i^2$ and $u_{i+1}^2$, respectively.  
Since $G$ is $K_{\ell,\ell,\ell}$-saturated, there is a subgraph $H$ of $G+v_iv_{i+1}$ that is isomorphic to $K_{\ell,\ell,\ell}.$  
It follows that $H$ must contain $v_i$, $v_{i+1}$ and $S_{i+2}$, and therefore $H$ cannot contain $u_i^2$ or $u_{i+1}^2$.
Since $H$ must contain $\ell$ neighbors of $v_i$ in $V_{i+1}$ and $u_{i+1}^2\notin H$, we conclude that $u_{i+1}^1\in H$.
Similarly, it follows that $u_i^1\in H$.
However, this implies that $H$ contains the nonedge $u_i^1u_{i+1}^1$, a contradiction.
Therefore, $G$ is isomorphic to a graph from Construction~\ref{con:1} or ~\ref{con:2}.
\end{proof}

\begin{proof}[Proof of Theorem~\ref{thm:lll-1}]
Let $G$ be a $K_{\ell,\ell,\ell-1}$-saturated subgraph of $K_{n_1,n_2,n_3}$ of minimum size.
It follows from Lemma~\ref{lem:mindeg} that if $\delta_i>2(\ell-1)$ for any $i\in[3]$, then $|E(G)|\ge 2(\ell-1)(n_1+n_2+n_3)$.
It is clear that $\delta(G)\ge 2(\ell-1)$, and thus we assume that $\delta_1=\delta_2=\delta_3=2(\ell-1)$.

For $i\in [3]$, let $v_i\in V_i$ be a vertex of degree $2(\ell-1)$. 
Thus $v_i$ has $\ell-1$ neighbors in $V_{i+1}$ and $\ell-1$ neighbors in $V_{i+2}$, and $G$ contains all edges joining $N_{i+1}(v_i)$ to $V_{i+2}\setminus N_{i+2}(v_i)$ and all edges joining $N_{i+2}(v_i)$ to $V_{i+1}\setminus N_{i+1}(v_i)$.
Therefore, the vertices of degree $2(\ell-1)$ in $G$ form an independent set.
Let $S=N(v_1)\cup N(v_2)\cup N(v_3)$ and let $S_i=S\cap V_i$.  
Since $v_{i+1}$ and $v_{i+2}$ have $\ell-1$ common neighbors, we conclude that $N_i(v_{i+1})=N_i(v_{i+2})$ and therefore $|S_i|=\ell-1.$
Furthermore, since the addition of an edge joining $v_i$ and a vertex in $V_{i+1}\setminus N_{i+1}(v_i)$ yields a copy of $K_{\ell,\ell,\ell-1}$, it follows that $N_{i+1}(v_i)$ and $N_{i+2}(v_i)$ must be completely joined.
Thus, $S_i$ and $S_{i+1}$ are completely joined for all $i\in [3]$.
Therefore the graph from Construction~\ref{con:4} is a subgraph of $G$.
Since $G$ is $K_{\ell,\ell,\ell-1}$-saturated, it follows that $G$ is isomorphic to the graph from Construction~\ref{con:4}, and therefore $\sat(K_{n_1,n_2,n_3},K_{\ell,\ell,\ell-1})=2(\ell-1)(n_1+n_2+n_3)-3(\ell-1)^2$.
\end{proof}

We note that it is possible to lower the bounds on $n_3$ in Theorems~\ref{thm:lll} and~\ref{thm:lll-1} through a more careful analysis of the algebra in Lemmas~\ref{lem:deltai} and~\ref{lem:mindeg}.
However, this appears still to yield a lower bound on $n_3$ that is cubic in $\ell$, and mainly distracts from the main ideas of the proof.

\section{The saturation number of $K_{\ell,\ell,\ell-2}$}~\label{sec4}

In this section we prove that the graph from Construction~\ref{con:5} is within an additive constant of the minimum number of edges in a $K_{\ell,\ell,\ell-2}$-saturated subgraph of $K_{n,n,n}$.
Given two sets of vertices $S$ and $T$, we let $[S,T]$ denote the set of edges with one endpoint in $S$ and one endpoint in $T$.

\begin{theorem}
Let $\ell$ be a positive integer.
For $n$ sufficiently large,
$$\sat(K_{n,n,n},K_{\ell,\ell,\ell-2})\ge 6(\ell-1)n-(72\ell^2-40\ell+54).$$
\end{theorem}

\begin{proof}

Let $G$ be a $K_{\ell,\ell,\ell-2}$-saturated subgraph of $K_{n,n,n}$.
If $\delta_i(G)\ge 6(\ell-1)$ for some $i\in [3]$, then $|E(G)|\ge 6(\ell-1)n$.
Therefore we may assume that $\delta_i<6(\ell-1)$ for all $i\in [3]$, and consequently a vertex of degree $\delta_i$ in $V_i$ must have nonneighbors in both $V_{i+1}$ and $V_{i+2}$.
Assume that $v$ is a vertex of degree at most $2\ell-3$ in $V_i$.
If $|N_{i+1}(v)|<\ell-2$, the the addition of an edge joining $v$ and $V_{i+2}$ does not complete a copy of $K_{\ell,\ell,\ell-2}$.
Therefore we may assume without loss of generality that $2\ell-4\le d(v)\le 2\ell-3$ and $v$ has $\ell-2$ neighbors in $V_{i+1}$ and at most $\ell-1$ neighbors in $V_{i+2}$.
It follows that the addition of an edge joining $v$ and $V_{i+1}$ does not complete a copy of $K_{\ell,\ell,\ell-2}$, and therefore $G$ is not $K_{\ell,\ell,\ell-2}$-saturated.
We conclude that $\delta_i\ge 2\ell-2$ for all $i\in [3]$.

Let $c=72\ell^2-40\ell+54$.
If $|[V_{i},V_{i+1}]|\ge 2(\ell-1)n-c/3$ for all $i\in[3]$, then $|E(G)|\ge 6(\ell-1)n-c$.
Therefore we may assume that $|[V_{i+1},V_{i+2}]|<2(\ell-1)n-c/3$ for some $i\in [3]$.
Let $v_i\in V_i$ have degree $\delta_i$.
Every vertex in $V_{i+1}\setminus N_{i+1}(v_i)$ is adjacent to at least $\ell-2$ vertices in $N_{i+2}(v_i)$.
If $v'$ is a vertex in $V_i$ that has only $\ell-2$ neighbors in $V_{i+2}$, then each vertex in $V_{i+2}\setminus N_{i+2}(v')$ has $\ell$ neighbors in $N_{i+1}(v')$.
Therefore 
\begin{align*}
|[V_{i+1},V_{i+2}]|&\ge (\ell-2)(n-\delta_i)+\ell(n-\delta_i-\ell+2)\\
&\ge 2(\ell-1)n-((2\ell-2)\delta_i+\ell^2-2\ell)\\
&\ge 2(\ell-1)n-(13\ell^2-26\ell+12),
\end{align*}
a contradiction.
Therefore we assume that every vertex in $V_i$ has at least $\ell-1$ neighbors in $V_{i+2}$, and by symmetry, also in $V_{i+1}$.

Let $X_i^0=N(v_i)$.
For $k\ge 1$, recursively define $X_i^k$ to be the vertices in $(V_{i+1}\cup V_{i+2}) \setminus(\bigcup_{j=0}^{k-1}X_i^j)$ that have at least $\ell-1$ neighbors in $\bigcup_{j=0}^{k-1}X_i^j$.
Define $X_i$ to be the set of vertices that are in $X_i^k$ for any value of $k$.
By definition, $G[X_i]$ contains at least $(\ell-1)(|X_i|-\delta_i)$ edges.

Let $R_i=(V_{i+1}\cup V_{i+2})\setminus X_i$.
Note that each vertex in $R_i$ is adjacent to exactly $\ell-2$ vertices in $N(v_i)$.
Let $T_{i,1},\ldots,T_{i,{a_i}}$ be the components of $G[R_i]$ that are trees.
Thus $G[R_i]$ contains at least $|R_i|-a_i$ edges, and 
\begin{equation}\label{ineq:VV}
|[V_{i+1},V_{i+2}]|\ge (\ell-1)(2n-\delta_i)-a_i\ge 2(\ell-1)n-6(\ell-1)^2-a_i.
\end{equation}
If $T_{i,b}$ consists of single vertex $v\in V_{i+1}$ and $T_{i,b'}$ consists of a single vertex $u\in V_{i+2}$, then the addition of $uv$ cannot complete a copy of $K_{\ell,\ell,\ell-2}$ in $G$.
Therefore, since $N_{i+1}(v_i)$ and $N_{i+2}(v_i)$ are nonempty,
\begin{equation}\label{a_i}
a_i\le \max\{|R_i\cap V_{i+1}|,|R_i\cap V_{i+2}|\}< n.
\end{equation}

Observe that
\begin{align*}
|E(G)|&\ge \sum_{j=1}^{a_i}\left(|E(T_{i,j})|+|[V(T_{i,j}),V_i]|\right).
\end{align*}
If $|E(T_{i,j})|+|[V(T_{i,j}),V_i]|> 6(\ell-1)n/a_i$ for all $j\in[a_i]$, then $|E(G)|> 6(\ell-1)n$.
Therefore we assume that there is a component $T_{i,k_i}$ of $G[R_i]$ such that $|E(T_{i,k_i})|+|E(T_{i,k_i},V_i)|\le 6(\ell-1)n/a_i$.
Thus $|V(T_{i,k_i})|\le 6(\ell-1)n/a_i+1$.
If $x\in V_{i+2}\cap V(T_{i,k_i})$ and $w\in V_{i+1}\setminus V(T_{i,k_i})$, then the addition of $xw$ cannot complete a copy of $K_{\ell,\ell}$ in $V_{i+1}\cup V_{i+2}$.
Therefore each vertex in $w\in V_{i+1}\setminus V(T_{i,k_i})$ has at least $\ell$ neighbors in $N_i(x)$.
Observe that $|N_i(x)|\le 6(\ell-1)n/a_i$.
Similarly, for $x\in V_{i+1}\cap V(T_{i,k_i})$, every vertex in $V_{i+2}\setminus V(T_{i,k_i})$ has at least $\ell$ neighbors in $N_i(x)$, and $|N_i(x)|\le 6(\ell-1)n/a_i$.
We consider two cases.

{\bf Case 1:} {\it For some $i\in 3$, $|[V_{i+1},V_{i+2}]|<2(\ell-1)n-c/3$ and $T_{i,k_i}$ contains vertices in both $V_{i+1}$ and $V_{i+2}$.}
Let $x_{i+1}\in V_{i+1}\cap V(T_{i,k_i})$ and let $x_{i+2}\in V_{i+2}\cap V(T_{i,k_i})$.
Therefore 
\begin{align*}
\sum_{v\in V_i}d(v)&\ge \delta_i(n-d_i(x_{i+1})-d_i(x_{i+2}))+\ell(n-d_{i+2}(x_{i+1})) +\ell(n-d_{i+1}(x_{i+2}))\\
&\ge 2(\ell-1)(n-12(\ell-1)n/a_i)+2\ell(n-6(\ell-1)n/a_i)
\end{align*}
Summing the edges we have
\begin{align*}
|E(G)|&\ge |[V_{i+1},V_{i+2}]|+\sum_{v\in V_i}d(v)\\
&\ge 2(\ell-1)n-6(\ell-1)^2-a_i+2(\ell-1)(n-12(\ell-1)n/a_i)+2\ell(n-6(\ell-1)n/a_i)\\
&\ge -a_i+(6(\ell-1)+2)n-6(\ell-1)^2-(36\ell^2-60\ell+24)n/a_i.
\end{align*}
If $|E(G)|<6(\ell-1)n$, then we conclude that 
\begin{align*}
a_i&<(n-3(\ell-1)^2)-\sqrt{(n-3(\ell-1)^2)^2-(36\ell^2-60\ell+24)n}\quad\text{or}\\
a_i&>(n-3(\ell-1)^2)+\sqrt{(n-3(\ell-1)^2)^2-(36\ell^2-60\ell+24)n}.
\end{align*}
From~(\ref{a_i}) we know that $a_i<n$, so we conclude that for $n$ sufficiently large, $$a_i<(n-3(\ell-1)^2)-\sqrt{(n-3(\ell-1)^2)^2-(36\ell^2-60\ell+24)n}.$$
Since 
$$\lim_{n\to\infty}(n-3(\ell-1)^2)-\sqrt{(n-3(\ell-1)^2)^2-(36\ell^2-60\ell+24)n}=18\ell^2-30\ell+12,$$
it follows from the integrality of $a_i$ that for $n$ sufficiently large, $a_i\le 18\ell^2-30\ell+12$.
Therefore $|[V_{i+1},V_{i+2}]|\ge 2(\ell-1)n-6(\ell-1)^2-(18\ell^2-30\ell+12)\ge 2(\ell-1)n-c/3$, a contradiction.

{\bf Case 2:} {\it For some $i\in 3$, $|[V_{i+1},V_{i+2}]|<2(\ell-1)n-c/3$ and $T_{i,k_i}\cap V_{i+1}=\varnothing$ or $T_{i,k_i}\cap V_{i+2}=\varnothing$.}
Without loss of generality we assume that $|[V_2,V_3]|<2(\ell-1)n-c/3$ and $T_{1,k_1}\cap V_3=\varnothing$.
Thus $T_{1,k_1}$ consists of a single vertex in $V_2$ that has only $\ell-2$ neighbors in $V_3$; call this vertex $x$.
Furthermore, $d(x)\le 6(\ell-1)n/a_1$.
Since the addition of an edge joining $x$ to $V_3$ cannot complete a copy of $K_{\ell,\ell}$ in $V_2\cup V_3$, each nonneighbor of $x$ in $V_3$ has at least $\ell$ neighbors in $N_1(x)$.
Since every vertex in $V_1$ has at least $\ell-1$ neighbors in $V_3$, we conclude that $|[V_1,V_3]|\ge (2\ell-1)(n-6(\ell-1)n/a_1)$.
Consequently,
\begin{align*}
|E(G)|&=|[V_1,V_2]|+|[V_1,V_3]|+|[V_2,V_3]|\\
&\ge |[V_1,V_2]|+(2\ell-1)(n-6(\ell-1)n/a_1)+(2(\ell-1)n-6(\ell-1)^2-a_1)\\
&=|[V_1,V_2]|+4(\ell-1)n+n-(12\ell^2-18\ell+6)n/a_1-6(\ell-1)^2-a_1.
\end{align*}

First assume that $|[V_1,V_2]|\ge 2(\ell-1)n-c/3$.
If $|E(G)|<6(\ell-1)n-c$, then
$$0\ge -a_1+n-6(\ell-1)^2+2c/3-(12\ell^2-18\ell+6)n/a_1,$$
which requires
\begin{align}
a_1&<\frac 12\left(n-6(\ell-1)^2+2c/3-\sqrt{(n-6(\ell-1)^2+2c/3)^2-(48\ell^2-72\ell+24)n}\right) \quad\text{or}\\
a_1&>\frac 12\left(n-6(\ell-1)^2+2c/3+\sqrt{(n-6(\ell-1)^2+2c/3)^2-(48\ell^2-72\ell+24)n}\right)\label{ineq:lb}.
\end{align}
Since $c\ge 45\ell^2-72\ell+27$, it follows that $2c/3\ge 30\ell^2-48\ell+18\ge 24\ell^2-36\ell+12+6(\ell-1)^2$.
Therefore, if inequality~(\ref{ineq:lb}) holds, then $a_1\ge n$.
This violates inequality~(\ref{a_i}), so we conclude that
$$a_1<\frac 12\left(n-6(\ell-1)^2+2c/3-\sqrt{(n-6(\ell-1)^2+2c/3)^2-(48\ell^2-72\ell+24)n}\right).$$
Since 
$$\lim_{n\to \infty}\frac {n-6(\ell-1)^2+\frac23c-\sqrt{\left(n-6(\ell-1)^2+\frac23c\right)^2-(48\ell^2-72\ell+24)n}}{2}= 12\ell^2-18\ell+6,$$
it follows from the integrality of $a_1$ that for $n$ sufficiently large, $a_1\le12\ell^2-18\ell+6$.
Therefore $|[V_2,V_3]|\ge 2(\ell-1)n-6(\ell-1)^2-(12\ell^2-18\ell+6)\ge 2(\ell-1)n-c/3$, a contradiction.

Now assume that $|[V_1,V_2]|< 2(\ell-1)n-c/3$.
Therefore $T_{3,k_3}$ exists.
If $T_{3,k_3}$ contains vertices in both $V_1$ and $V_2$, then by Case 1 we conclude that $|E(G)|\ge 6(\ell-1)n-c$.
Therefore we assume that $T_{3,k_3}$ contains a single vertex $y\in V_1\cup V_2$, and $d(y)\le 6(\ell-1)n/a_3$.
Since every vertex in $V_1$ has at least $\ell-1$ neighbors in both $V_2$ and $V_3$ and $y$ has only $\ell-2$ neighbors in $V_1\cup V_2$, we conclude that $y\in V_2$.

The $n-(\ell-2)$ nonneighbors of $x$ in $V_3$ each have at least $\ell$ neighbors in $N_1(x)$.
Similarly, each vertex in $V_1\setminus (N_1(x)\cup N_1(y))$ has at least $\ell$ neighbors in $N_3(y)$.
Since $|V_1\setminus (N_1(x)\cup N_1(y))|\ge n-6(\ell-1)n/a_1-(\ell-2)$, we conclude that
$$|[V_1,V_3]|\ge 2\ell n-6\ell(\ell-1)n/a_1-2\ell(\ell-2).$$
Using inequalities~(\ref{ineq:VV}) and~(\ref{a_i}), we have
\begin{align*}
|E(G)|&=|[V_1,V_3]|+|[V_2,V_3]|+|[V_1,V_2]|\\
&\ge (2\ell n-6\ell(\ell-1)n/a_1-2\ell(\ell-2))+(4(\ell-1)n-12(\ell-1)^2-a_1-a_3)\\
&\ge -a_1+6(\ell-1)n+2n-a_3-(14\ell^2-28\ell+12)-6\ell(\ell-1)n/a_1\\
&\ge -a_1+6(\ell-1)n+n-(14\ell^2-28\ell+12)-6\ell(\ell-1)n/a_1.
\end{align*}
Therefore $|E(G)|<6(\ell-1)n-c$ only if 
\begin{align}
a_1&< \frac12\left(n+c-(14\ell^2-28\ell+12)-\sqrt{(n+c-(14\ell^2-28\ell+12))^2-24\ell(\ell-1)n}\right) \quad\text{or}\\
a_1&> \frac12\left(n+c-(14\ell^2-28\ell+12)+\sqrt{(n+c-(14\ell^2-28\ell+12))^2-24\ell(\ell-1)n}\right).\label{ineq:lb2}
\end{align}
Since $c\ge 26\ell^2-40\ell+12$, it follows that $c-(14\ell^2-28\ell+12)\ge 12\ell(\ell-1)$.
Therefore, if inequality~(\ref{ineq:lb2}) holds, then $a_1\ge n$.
This violates inequality~(\ref{a_i}), so we conclude that
$$a_1< \frac12\left(n+c-(14\ell^2-28\ell+12)-\sqrt{(n+c-(14\ell^2-28\ell+12))^2-24\ell(\ell-1)n}\right).$$
Since 
$$\lim_{n\to \infty}\frac{\left(n+c-(14\ell^2-28\ell+12)-\sqrt{(n+c-(14\ell^2-28\ell+12))^2-24\ell(\ell-1)n}\right)}2 =6\ell(\ell-1),$$
it follows from the integrality of $a_1$ that for $n$ sufficiently large, $a_1\le 6\ell(\ell-1)$.
Therefore $|[V_2,V_3]|\ge 2(\ell-1)n-6(\ell-1)^2-6\ell(\ell-1)\ge 2(\ell-1)n-c/3$, a contradiction.
\end{proof}

\section{Conclusion}\label{sec5}

We conclude with several open questions and conjectures.
First, we conjecture that in a sufficiently large, sufficiently unbalanced host graph, the constructions in Section{~\ref{sec:con} are best possible.

\begin{conjecture}
Let $\ell$ and $m$ be positive integers such that $\ell>m$.
For $n_1\ge n_2\ge n_3$, $n_3$ sufficiently large compared to $\ell$, and $n_1$ sufficiently large compared to $n_3$, 
\[ \sat(K_{n_1,n_2,n_3},K_{\ell,m,m})= 2m(n_1+n_2+n_3)+(\ell-m)(n_2+2n_3)-3\ell m-3.\]
\end{conjecture}

\begin{conjecture}
Let $\ell$, $m$, and $p$ be positive integers such that $\ell\ge m>p$.
For $n_1\ge n_2\ge n_3$, $n_3$ sufficiently large compared to $\ell$, and $n_1$ sufficiently large compared to $n_3$, 
\[ \sat(K_{n_1,n_2,n_3},K_{\ell,m,p})= 2(m-1)(n_1+n_2+n_3)+(\ell-m)(n_2+2n_3)-3\ell(m-1)+3m-3.\]
\end{conjecture}

Following the direction taken in~\cite{FJPW}, one can study the saturation number of $K_{\ell,m,p}$ in $k$-partite graphs for $k>3$.
The following is the logical place to begin such research.

\begin{question}
Let $K_k^n$ denote the complete $k$-partite graph in which all partite sets have size $n$.
For $\ell\ge 2$, $k\ge 4$, and $n$ sufficiently large, what is $\sat(K_k^n,K_{\ell,\ell,\ell})$?
\end{question}

We also note that if $G$ is a graph with chromatic number at most $3$, then determining $\sat(K_{n_1,n_2,n_3},G)$ is nontrivial.
Thus it is natural to consider the saturation number of bipartite graphs in complete tripartite graphs.
As a first example, we compute the saturation number of $C_4$ in tripartite graphs.

\begin{proposition}
For $n_1\ge n_2\ge n_3\ge 2$,
$$\sat(K_{n_1,n_2,n_3},C_4)=n_1+n_2+n_3.$$
\end{proposition}

\begin{proof}
It is clear that a $C_4$-saturated subgraph of $K_{n_1,n_2,n_3}$ must be connected, and no spanning tree of $K_{n_1,n_2,n_3}$ is $C_4$-saturated.
It is also straightforward to check that the graph with edge set $\{v_i^1v_{i+1}^j|i\in [3], j\in[n_{i+1}]\}$ is $C_4$-saturated (see Figure~\ref{fig:c4}).
\end{proof}

\begin{figure}[h]
\centering
\begin{tikzpicture}[x=.6cm,y=.6cm]
\path
(0,1.75) node [shape=ellipse,minimum width=1.5cm, minimum height=.75cm,draw](V1){$V_1-v_1^1$}
(-5.5,-5.196) node [shape=ellipse,minimum width=1.5cm, minimum height=.75cm,draw](V3){$V_3-v_3^1$}
(5.5,-5.196) node  [shape=ellipse,minimum width=1.5cm, minimum height=.75,draw](V2){$V_2-v_2^1$}

(0,-1.7) node [draw=none](n1){$v_1^{1}$}
(1.4,-4.3) node [draw=none](n2){$v_2^{1}$}
(-1.5,-4.3) node [draw=none](n3){$v_3^{1}$};
\fill (0,-.75) circle (3pt)
(2.4,-4.75)circle (3pt) 
(-2.4,-4.75)circle (3pt);
\draw (0,-.75)to [bend left](V2);
\draw (2.4,-4.75)to [bend left](V3);
\draw (-2.4,-4.75)to [bend left](V1);
\draw (0,-.75) to (2.4,-4.75) to (-2.4,-4.75) to(0,-.75);
\end{tikzpicture}
\caption[Construction~\ref{con:1}, a $K_{\ell,m,m}$-saturated subgraph of $K_{n_1,n_2,n_3}$]{\label{fig:c4} A $C_4$-saturated subgraph of $K_{n_1,n_2,n_3}$.  Solid lines denote complete joins between two sets.}
\end{figure}
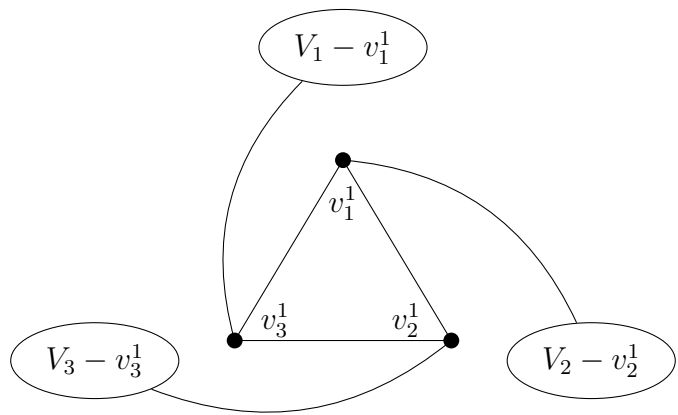

Observe that $\sat(K_{n_1,n_2,n_3},C_4)$ and the sharpness example are not obtained using the bipartite saturation number of $C_4$.
Thus it appears that the study of saturation numbers of bipartite graphs in tripartite graphs will differ from the work initiated in~\cite{GKS} and~\cite{MS}.

\end{document}